\documentclass[12pt]{amsart}
\usepackage{amsmath,amssymb,amsbsy,amsfonts,amsthm,latexsym,
                        amsopn,amstext,amsxtra,euscript,amscd,mathrsfs,color}

\usepackage{float}
\usepackage[english]{babel}
\usepackage{url}
\usepackage[colorlinks,linkcolor=blue,anchorcolor=blue,citecolor=blue]{hyperref}

\newtheorem{theorem}{Theorem}
\newtheorem{lemma}[theorem]{Lemma}

\def\A{\mathbb{A}}
\def\F{\mathbb{F}}
\def\Z{\mathbb{Z}}
\def\ord{\text{\rm ord}}
\def\sgn{\text{\rm sgn}}

\numberwithin{equation}{section}
\numberwithin{theorem}{section}

\begin{document}

\title{Gauss factorials of polynomials over finite fields}
%\markright{Abbreviated Article Title}
\author{Xiumei Li}
\address{School of Mathematical Sciences, Qufu Normal University, Qufu Shandong, 273165, China}
\email{lxiumei2013@hotmail.com}

\author{Min Sha}
%(hchoover@stanford.edu) entered Stanford University in 1891, after failing all of the entrance exams except mathematics.  He received his BS degree in geology in 1895, spent time as a mining engineer, then was appointed by his co-author to the U.S. Food Administration and the Supreme Economic Council, where he orchestrated the greatest famine relief efforts of all time.
\address{School of Mathematics and Statistics, University of New South Wales,
 Sydney, NSW 2052, Australia}
\email{shamin2010@gmail.com}

\begin{abstract}
In this paper we initiate a study on Gauss factorials of polynomials over finite fields, which are the analogues of Gauss factorials of positive integers.
\end{abstract}

\maketitle

\section{Introduction}

A well-known result in number theory, called
Wilson's theorem, says that for any prime number $p$, we have
\begin{equation}
\label{eq:Wilson}
(p-1)! \equiv -1 \pmod p.
\end{equation}
If we replace $p$ by any composite integer in \eqref{eq:Wilson}, it is not correct any more.
Gauss proved a composite analogue of Wilson's theorem: for any integer $n>1$,
\begin{equation}
\label{eq:Gauss}
\begin{split}
\prod_{\substack{1\le j \le n-1 \\ \gcd(j,n)=1}} j
& \equiv \left\{\begin{array}{ll}
-1 \pmod n & \textrm{for $n=2,4,p^k,$ or $2p^k$,}\\
\\
1 \pmod n &\textrm{otherwise},
\end{array}
\right.
\end{split}
\end{equation}
where $p$ is an odd prime and $k$ is a positive integer.

In \cite{CD2008}, Cosgrave and Dilcher called the product
\begin{equation}
\label{eq:CD}
\prod_{\substack{1\le j \le N \\ \gcd(j,n)=1}} j
\end{equation}
a \emph{Gauss factorial}, where $N$ and $n$ are positive integers.
We refer to \cite{CD2011} for a very good survey on Gauss factorials of integers, and \cite{CD2010, CD2011a, CD2014, CD2015} for recent work. In particular, when $n\ge 3$ is odd and $N=(n-1)/2$, the multiplicative order of the Gauss factorial \eqref{eq:CD} modulo $n$ has been determined in \cite[Theorem 2]{CD2008}.

Throughout the paper, the letter $p$ always denotes a prime, and $q=p^s$ for some integer $s\ge 1$. Let $\F_q$ be the finite field of $q$ elements. We denote by $\F_q[X]$ the polynomial ring over $\F_q$.
In the sequel, for simplicity let $\A=\F_q[X]$.
It is known that $\A$ has many properties in common with the integers $\Z$, and thus many number theoretic problems about $\Z$ have their analogues for $\A$.

For any polynomial $f(X)\in \A$ of degree $\deg f \ge 1$, we define the \textit{Gauss factorial} (denoted by $G(f)$) of $f$ as follows:
$$
G(f) = \prod_{\substack{g\in \A \\ 0 \le \deg g < \deg f \\ \gcd(g,f)=1}} g.
$$
That is, $G(f)$ is the product of all \emph{non-zero} polynomials of degree less than $\deg f$ and coprime to $f$.
By convention, if $f\in \F_q$, we define $G(f)=1$.
Note that if $f$ is irreducible, the definition of $G(f)$ only depends on the degree of $f$.

Moreover, given integer $n\ge 1$ and polynomial $f\in \A$, we define
$$
G(n,f) = \prod_{\substack{g\in \A \\ 0 \le \deg g \le n \\ \gcd(g,f)=1}} g,
$$
which is also called a \textit{Gauss factorial}. In particular, $G(f)=G(\deg f - 1, f)$.

In this paper, we initiate the study of Gauss factorials of polynomials over finite fields by generalizing  \eqref{eq:Gauss} and the main result in \cite{CD2008}.
Certainly, there are many other things remaining to be explored.

We want to remark that one can similarly define \emph{factorials} of polynomials over finite fields, and consider
generalizing related classical results (some of them are listed in \cite{Bhargava}).

\section{Preliminaries}

In this section, for the convenience of the reader we recall some basic results about polynomials over finite fields,
which can be found in \cite[Chapter 1 and Chapter 3]{Rosen}.

 For $f\in \A$, set $|f|=q^{\deg f}$ if $f\ne 0$, and $|f|=0$ otherwise.
For a non-constant polynomial $f\in \A$, write its prime factorization as
\begin{equation}
\label{eq:factor}
f= a P_1^{e_1} \cdots P_t^{e_t},
\end{equation}
where $a \in \F_q^*$, integer $e_j \ge 1$ ($1\le j \le t$), and each $P_j$ ($1\le j \le t$) is a monic irreducible polynomial. Here,  a monic irreducible polynomial in $\A$ is said to be a \emph{prime polynomial}, and so in \eqref{eq:factor} each $P_j$
$(1\le j \le t)$ is a \emph{prime divisor} of $f$.

Given a non-zero polynomial $f\in \A$, denote by $\A/f\A$ the residue class ring of $\A$ modulo $f$ and by $(\A/f\A)^*$ its unit group.

\begin{lemma}
\label{lem:structure1}
Let $f\in \A$ be a non-constant polynomial with prime factorization as in \eqref{eq:factor}.
 Then, we have
$$
(\A/f\A)^* \cong (\A/P_1^{e_1}\A)^* \times \cdots \times (\A/P_t^{e_t}\A)^*.
$$
\end{lemma}

\begin{lemma}
\label{lem:structure2}
Let $P\in \A$ be an irreducible polynomial, and $e$ a positive integer. Then, we have
$$
(\A/P^e\A)^* \cong (\A/P\A)^* \times (\A/P^e\A)^{(1)},
$$
where $(\A/P\A)^*$ is a cyclic group of order $|P|-1$, and $(\A/P^e\A)^{(1)}$ is a $p$-group of order $|P|^{e-1}$.
\end{lemma}

For any non-zero polynomial $f\in \A$, let $\Phi(f)= |(\A / f\A)^*|$, which is the so-called \emph{Euler totient function}
of $\A$.

\begin{lemma}
\label{lem:Euler}
For any non-zero polynomial $f\in \A$, we have
$$
\Phi(f) = |f| \prod_{P\mid f} \big( 1- 1/|P| \big),
$$
where the product runs through all the prime divisors of $f$.
\end{lemma}

Let $P\in \A$ be an irreducible polynomial.
When $q$ is odd, given a non-zero polynomial $g \in \A$ with $\gcd(g,P)=1$, as usual we define the \textit{Legendre symbol} $\left(\frac{g}{P}\right) = \pm 1$ such that
$$
\left(\frac{g}{P}\right) \equiv g^{\frac{|P|-1}{2}} \pmod P.
$$
Let $f\in \A$ be a non-constant polynomial with the prime factorization as in \eqref{eq:factor}. Given a non-zero polynomial $g \in \A$ with $\gcd(g,f)=1$, the \textit{Kronecker symbol} $\left(\frac{g}{f}\right) = \pm 1$ is defined as
$$
\left(\frac{g}{f}\right) = \prod_{j=1}^{t} \left(\frac{g}{P_j}\right)^{e_j}.
$$
As usual, for two non-zero polynomials $g_1, g_2 \in \A$ with $\gcd(g_1g_2, f)=1$, we have
$$
\left(\frac{g_1g_2}{f}\right) = \left(\frac{g_1}{f}\right) \left(\frac{g_2}{f}\right).
$$

\begin{lemma}[The reciprocity law]
\label{lem:reciprocity}
Let $f,g \in \A$ be relatively prime non-zero polynomials. Assume that $q$ is odd, and both $f$ and $g$ are monic. Then, we have
$$
\left(\frac{g}{f}\right) = (-1)^{\frac{q-1}{2}\deg f \deg g} \left(\frac{f}{g}\right).
$$
\end{lemma}

The following lemma is essentially a special case of a result due to Artin \cite[Section 18, Equation (10)]{Artin}.

\begin{lemma}
\label{lem:Artin}
Assume that $q$ is odd. Let $P\in \A$ be a prime polynomial of odd degree. Denote by $h(-P)$ the class number of the quadratic function field $\F_q(X, \sqrt{-P})$. Then, we have
$$
h(-P)= \sum_{\substack{\textrm{$g$ monic} \\ 0 \le \deg g < \deg P}}
\left(\frac{g}{P}\right).
$$
\end{lemma}

\begin{proof}
Following Artin \cite{Artin} the quadratic function field $\F_q(X, \sqrt{-P})$ is imaginary
(see also \cite[Proposition 14.6]{Rosen} and the discussions therein).
So, using \cite[Section 18, Equation (10)]{Artin} and \cite[Section 17, Equation (4)]{Artin} directly
(alternatively, applying arguments similar to those in the proof of \cite[Theorem 16.8]{Rosen}),
we obtain
$$
h(-P)= \sum_{\substack{\textrm{$g$ monic} \\ 0 \le \deg g < \deg P}}
\left(\frac{-P}{g}\right).
$$
Here one should note that since $-P$ can be viewed as a monic polynomial with respect to $-X$, we can apply the results in \cite{Artin}.

For each summation term in the above formula, using the reciprocity law (Lemma \ref{lem:reciprocity}) and noticing that $\deg P$ is odd, we deduce that
$$
\left(\frac{-P}{g}\right) = (-1)^{\frac{(q-1)\deg g}{2}} \left(\frac{P}{g}\right) = \left(\frac{g}{P}\right),
$$
which implies the desired result.
\end{proof}

Finally, we reproduce a useful result due to Miller \cite[Section 1]{Miller}, which can yield a simple proof of \eqref{eq:Gauss}.

\begin{lemma}
\label{lem:Miller}
Let $A$ be an abelian group. Then, the product $\prod_{a \in A} a$
is the identity element if $A$ either has no element of order two or has more than one element of order two,
otherwise the product is equal to the element of order two.
\end{lemma}

\section{Congruence formulas}

In this section, our main objective is to generalize \eqref{eq:Gauss} to $G(f)$ by following the approach in \cite{Miller}.
One can see that we have two quite different cases depending on
the characteristic of $\A$ (that is, $p$).

We first remark that it is known that for any irreducible polynomial $P\in \A$, we have
(for instance see \cite[Chapter 1, Corollary 2]{Rosen})
$$
G(P) \equiv -1 \pmod P,
$$
which is an analogue of \eqref{eq:Wilson}. Besides, it is easy to see that
if non-zero $f\in \A$ is reducible and square-free, then we have
$$
\prod_{\substack{g\in \A \\ 0\le \deg g < \deg f}} g \equiv 0 \pmod f.
$$

\begin{theorem}
\label{thm:con1}
Assume that $p$ is odd. Then, for any polynomial $f \in \A$ of degree $\deg f \ge 1$, we have
\begin{equation*}
\begin{split}
G(f)
& \equiv \left\{\begin{array}{ll}
-1 \pmod f & \textrm{if $f$ has only one prime divisor,}\\
\\
1 \pmod f &\textrm{otherwise}.
\end{array}
\right.
\end{split}
\end{equation*}
\end{theorem}

\begin{proof}
Note that $p$ is odd. Then, for any irreducible polynomial $P\in \A$ and integer $e\ge 1$,
by Lemma \ref{lem:structure2} the group $(\A/P^e\A)^*$ has only one element of order two (that is, $-1$).
So, applying Lemma \ref{lem:structure1}, it is easy to see that
the abelian group $(\A/f\A)^*$ has only one element of order two if and only if $f$ has only one prime divisor. Then, the desired result now follows from Lemma \ref{lem:Miller}.
\end{proof}

\begin{theorem}
\label{thm:con2}
Assume that $p=2$. For any polynomial $f\in \A$ of degree $\deg f \ge 1$ with the prime factorization as in \eqref{eq:factor},
 we have
\begin{equation*}
\begin{split}
G(f)
& \equiv \left\{\begin{array}{ll}
f/P_1 +1 \pmod f & \textrm{if $q=2, \deg P_1=1, 2\le e_1 \le 3$, and all the }\\
&  \quad \textrm{other exponents $e_j=1$ if they exist, }\\
1 \pmod f &\textrm{otherwise}.
\end{array}
\right.
\end{split}
\end{equation*}
\end{theorem}

\begin{proof}
We first remark that $-1$ is not an element of order two, because $p=2$.
We also note that for any irreducible polynomial $P\in \A$, the order of
 $(\A/P\A)^*$ is an odd number, and thus the group $(\A/P\A)^*$ has no element of order two.

Let $P\in \A$ be an irreducible polynomial, and $e$ a positive integer.
If $h$ is an element of order two of $(\A/P^e\A)^*$, then $P^e \mid (h+1)^2$.
So, it is easy to see that when $e$ is even, the set of elements of order two of $(\A/P^e\A)^*$ is
$$
\Big\{gP^{e/2} +1:\, g\in \A, 0\le \deg g < \frac{e}{2}\deg P \Big\},
$$
whose cardinality is $q^{\frac{1}{2}e\deg P}-1$. Similarly, if $e$ is odd and greater than 1, then
the set of elements of order two of $(\A/P^e\A)^*$ is
$$
\Big\{gP^{(e+1)/2} +1:\, g\in \A, 0\le \deg g < \frac{e-1}{2}\deg P \Big\},
$$
whose cardinality is $q^{\frac{1}{2}(e-1)\deg P}-1$.
Thus, $(\A/P^e\A)^*$ has only one element of order two if and only if $\deg P=1, 2\le e \le 3$, and $q=2$.

Hence, the desired result follows by using Lemma \ref{lem:structure1} and Lemma \ref{lem:Miller}, and noticing
that $f/P_1 +1$ is indeed an element of order two of $(\A/f\A)^*$ if $e_1 \ge 2$.
\end{proof}

Furthermore, we can get a congruence identity for $G(n,f)$ when $n \ge \deg f$. This can also be viewed as an analogue of \eqref{eq:Gauss}.

\begin{theorem}
\label{thm:con3}
For any polynomial $f\in \A$ of degree $\deg f \ge 1$,  if the integer $n$ satisfies $n\ge \deg f$, we have
\begin{equation*}
\begin{split}
G(n,f)
& \equiv \left\{\begin{array}{ll}
-1 \pmod f & \textrm{if $p$ is odd and $f$ has only one prime divisor,}\\
\\
1 \pmod f &\textrm{otherwise}.
\end{array}
\right.
\end{split}
\end{equation*}
\end{theorem}

\begin{proof}
Fix an arbitrary integer $m \ge \deg f$. For any polynomial $g\in \A$ with $\deg g < \deg f$ and $\gcd(g,f)=1$,
it is easy to see that the set of polynomials in $\A$ of degree $m$ and congruent to $g$ modulo $f$ is
$$
\Big\{g+fr: r\in \A, \deg r = m - \deg f  \Big\},
$$
whose cardinality is $(q-1)q^{m- \deg f}$.

Thus, we obtain
\begin{align*}
G(n,f) & \equiv G(f)\prod_{m= \deg f}^{n} G(f)^{(q-1)q^{m- \deg f}} \\
& \equiv G(f)^{q^{n+1-\deg f}} \pmod f.
\end{align*}
We then conclude the proof by using Theorem \ref{thm:con1} and Theorem \ref{thm:con2}.
\end{proof}

\section{Multiplicative orders}

As mentioned before, when the integer $n\ge 3$ is odd and $N=(n-1)/2$, Cosgrave and Dilcher have determined the multiplicative order of the Gauss factorial \eqref{eq:CD} modulo $n$ in \cite[Theorem 2]{CD2008}. That is, they only considered \emph{half} of the positive integers less than $n$. In this section, assume that $q$ is odd, then our aim is to get some similar results concerning a special divisor of $G(f)$. For this divisor, we only consider \emph{half} of the polynomials of degree less than $\deg f$ and coprime to $f$.

For any non-zero $g \in \A$, we denote by $\sgn(g)$ the leading coefficient of $g$, which is called the \emph{sign} of $g$. Let $S$ be a subset of $\F_q^*$ such that
$|S|= (q-1)/2$ and for any $\alpha \in \F_q^*$ if $\alpha \in S$ then $-\alpha \not\in S$.
Obviously, the set $S$ has $2^{(q-1)/2}$ choices.

Define $\delta(S) = \prod_{a \in S} a.$
Notice that
\begin{equation}
\label{eq:delta}
\delta(S) ^ 2 = (-1)^{\frac{q-1}{2}} \prod_{a\in \F_q^*} a = (-1)^{\frac{q+1}{2}}.
\end{equation}
So, if $q\equiv 3 \pmod 4$, we have $\delta(S) ^ 2 = 1$, and then
$\delta(S)=\pm 1$.

Note that for any non-zero $g \in \A$,  $\sgn(g) \in S$ if and only if $\sgn(-g) \not \in S$. So, for any polynomial $f \in \A$ of degree $\deg f \ge 1$, we easily have
$$
G(f)
 = (-1)^{\Phi(f)/2} G(f,S)^2,
$$
where
$$
G(f,S) = \prod_{\substack{g\in \A, \sgn(g) \in S \\ 0 \le \deg g < \deg f \\ \gcd(g,f)=1}} g.
$$
Thus, by Theorem \ref{thm:con1} we obtain
\begin{equation*}
\begin{split}
G(f,S)^2
& \equiv \left\{\begin{array}{ll}
-(-1)^{\Phi(f)/2} \pmod f & \textrm{if $f$ has only one prime divisor,}\\
\\
(-1)^{\Phi(f)/2} \pmod f &\textrm{otherwise}.
\end{array}
\right.
\end{split}
\end{equation*}
This implies that the multiplicative order of $G(f,S)$ modulo $f$ only can possibly be 1, 2 or 4.

Here, we want to determine the multiplicative order of $G(f,S)$ modulo $f$, which is denoted by
$\ord_f G(f,S)$. The main result is as follows.

\begin{theorem}
\label{thm:extension}
Assume that $q$ is odd, and let $f \in \A$ be a polynomial having the prime factorization as in \eqref{eq:factor}. Then

\begin{enumerate}

\item[(1)]$\ord_f G(f,S)=4$ when $t=1$, and either $q \equiv 1 \pmod 4$ or $\deg P_1$ is even.

\item[(2)] $\ord_f G(f,S)=2$ when
\begin{enumerate}
\item $t=1$, $q \equiv 3 \pmod 4$, $\deg P_1$ is odd, $\delta(S)=1$, and $e_1 + \frac{1}{2}(h(-P_1)-3) \equiv 1 \pmod 2$, or

\item $t=1$, $q \equiv 3 \pmod 4$, $\deg P_1$ is odd, $\delta(S)=-1$, and $e_1 + \frac{1}{2}(h(-P_1)-3) \equiv 0 \pmod 2$, or

\item $t=2$, $q \equiv 1 \pmod 4$, and $P_1$ is not a quadratic residue modulo $P_2$, or

\item $t=2$, $q \equiv 3 \pmod 4$, both $\deg P_1$ and $\deg P_2$ are even, and $P_1$ is not a quadratic residue modulo $P_2$, or

\item $t=2$, $q \equiv 3 \pmod 4$, and either $\deg P_1$ or $\deg P_2$ is odd;
\end{enumerate}

\item[(3)] $\ord_f G(f,S)=1$ in all other cases.
\end{enumerate}
\end{theorem}

From Theorem \ref{thm:extension} one can see that for many cases $\ord_f G(f,S)$ is independent of the choice of $S$. Especially some results are related to the class numbers of function fields.
Actually, we do more in the paper: the value of $G(f,S)$ modulo $f$ is either explicitly given or easily computable.

Before proving Theorem \ref{thm:extension}, we illustrate an example to confirm some results in Theorem \ref{thm:extension} by using the computer algebra system PARI/GP \cite{Pari}. Choose $q=3$ and $P=X^3+2X+2$, then by Lemma \ref{lem:Artin} we have $h(-P)=7$. If furthermore we choose $S=\{1\}$, we have $\delta(S)=1$ and $G(P,S)=-1$, and thus $\ord_P G(P,S)=2$, which is compatible with Theorem \ref{thm:extension} (2); otherwise if we choose $S=\{-1\}$, we get $\ord_P G(P,S)=1$, which is also consistent with Theorem \ref{thm:extension} (3).

In the following, we divide the proof of Theorem \ref{thm:extension} into several cases.

\subsection{One prime divisor}

We continue the general discussion about $G(f,S)$.

Suppose that $f$ has the prime factorization as in \eqref{eq:factor}. Then, by Lemma \ref{lem:Euler} we get
$$
\Phi(f) = \prod_{i=1}^{t} q^{(e_i-1)\deg P_i} (q^{\deg P_i}-1).
$$
Note that $q$ is odd, so
$$
(-1)^{\frac{1}{2}\Phi(f)}= (-1)^{\frac{1}{2} \prod_{i=1}^{t} (q^{\deg P_i}-1)}.
$$

If $t\ge 2$, that is, $f$ has at least two distinct prime divisors, then we have
$(-1)^{\frac{1}{2}\Phi(f)}=1$, and thus
$$
G(f,S)^2 \equiv 1 \pmod f.
$$

Now assume that $t=1$, that is, $f$ has only one prime divisor $P_1$. Then it is easy to see that
\begin{equation*}
\begin{split}
G(f,S)^2
& \equiv \left\{\begin{array}{ll}
-1  & \textrm{if $q\equiv 1 \pmod 4$, or $\deg P_1$ is even,}\\
\\
1 &\textrm{otherwise}.
\end{array}
\right.
\end{split}
\end{equation*}

Thus, for any $f\in \A$ of positive degree and with the prime factorization as in \eqref{eq:factor}, we have
\begin{equation}
\label{eq:G2}
\begin{split}
G(f,S)^2
& \equiv \left\{\begin{array}{ll}
-1 \pmod f & \textrm{if $t=1$, and either $q\equiv 1$ (mod 4) or $\deg P_1$ is even,}\\
\\
1 \pmod f &\textrm{otherwise}.
\end{array}
\right.
\end{split}
\end{equation}

The above equation immediately gives the following partial result:

\begin{theorem}
\label{thm:order4}
Assume that $q$ is odd. If $f$ has only one prime divisor $P$, and either $q \equiv 1 \pmod 4$ or $P$ has even degree, then
$\ord_f G(f,S)=4$. Otherwise, $\ord_f G(f,S)=1$ or $2$.
\end{theorem}

For further deductions, we need to get a new expression of $G(f,S)$. Fix an arbitrary polynomial $f\in \A$ of degree $\deg f \ge 1$.
Let $i_{n}$ be the cardinality of the set $\{g\in \A: \, g \ \textrm{monic},\deg g =n, \gcd(g,f)=1\}$. Note that
$$
\sum_{n=0}^{\deg f-1} i_n = \frac{\Phi(f)}{q-1}.
$$

Denote 
$$
M(f) = \Big(  \prod_{\substack{g \textrm{ monic} \\ 0 \le \deg g < \deg f, \,\gcd(g,f)=1}} g \Big)^{(q-1)/2},
$$
which is also a polynomial in $\A$. Now, we deduce that
\begin{equation}
\label{eq:GM}
\begin{split}
G(f,S) & =  \prod_{a \in S} \Big( \prod_{n=0}^{\deg f -1 } a^{i_{n}} \prod_{\substack{ \textrm{$g$ monic} \\ \deg g =n, \, \gcd(g,f)=1}} g \Big) \\
& = \Big( \prod_{a\in S} a \Big)^{\frac{\Phi(f)}{q-1}}
\Big(  \prod_{\substack{g \textrm{ monic} \\ 0 \le \deg g < \deg f, \,\gcd(g,f)=1}} g \Big)^{\frac{q-1}{2}}\\
& = \delta(S)^{\frac{\Phi(f)}{q-1}}M(f).
\end{split}
\end{equation}
So, our problem is reduced to studying $M(f)$ modulo $f$.

By definition we also rewrite
\begin{align*}
G(f) = \prod_{\substack{0 \le \deg g < \deg f \\ \gcd(g,f)=1}} g
& = \prod_{n=0}^{\deg f -1 } \Big( \prod_{\substack{  \deg g =n \\ \gcd(g,f)=1}} g \Big) \\
& = \prod_{n=0}^{\deg f -1 } \left( \big( \prod_{a\in \F_q^*} a \big)^{i_n} \Big(\prod_{\substack{  \textrm{$g$ monic} \\ \deg g =n, \, \gcd(g,f)=1}} g \Big)^{q-1}\right) \\
& = (-1)^{\frac{\Phi(f)}{q-1}} M(f)^2.
\end{align*}
Now assume that $f$ has the prime factorization as in \eqref{eq:factor}.  So, using Theorem \ref{thm:con1} we obtain
\begin{equation}
\label{eq:M2}
\begin{split}
M(f)^2
& \equiv \left\{\begin{array}{ll}
-1 \pmod f & \textrm{if $t=1$ and $\deg P_1$ is even,}\\
\\
1 \pmod f &\textrm{otherwise}.
\end{array}
\right.
\end{split}
\end{equation}

We first handle $M(f)$ in the case when $f$ is irreducible.

\begin{lemma}
\label{lem:MP}
Assume that $q \equiv 3 \pmod 4$.
If $f \in \A$ is an irreducible polynomial of odd degree with the form $f=aP$, where $a \in \F_q^*$ and $P\in \A$ is a prime polynomial,   then
$$
M(f) \equiv  (-1)^{\frac{1}{2}(h(-P)-1)} \pmod P.
$$
\end{lemma}

\begin{proof}
We apply similar arguments as in \cite{Mordell}.
By definition, we directly have
$M(f)=M(P)$. So, it is equivalent to determine the value of $M(P)$ modulo $P$ (or equivalently modulo $f$).

 Put $d= \deg P$.
We first make some preparations. Let $N$ (resp. $R$) be the number of
monic polynomials in $\A$ of degree less than $d$ which are quadratic non-residues (resp. quadratic residues)
modulo $P$. So,
$$
N+R= 1+ q + \cdots + q^{d -1},
$$
which, together with Lemma \ref{lem:Artin}, implies that
$$
 N = \frac{1}{2} \Big(1+ q + \cdots + q^{d -1} - h(-P) \Big) .
$$
Note that $q \equiv 3 \pmod 4$ and $d$ is odd, so we have
$$
1+ q + \cdots + q^{d -1} \equiv 1 \pmod 4.
$$
Thus, we get
\begin{equation}
\label{eq:nonresidue}
N \equiv \frac{1}{2} \big( h(-P) - 1  \big) \pmod 2.
\end{equation}

Given a non-zero polynomial $g$ which is a quadratic residue modulo $P$, $bg$ is also a quadratic residue for any square element $b\in \F_q^*$ (equivalently, $b$ is a quadratic residue modulo $P$), and
$$
\prod_{b\in \F_q^*, \, \left(\frac{b}{P}\right)=1} bg =g^{(q-1)/2}.
$$
Besides, note that for any non-zero polynomial $g\in \A$, we can write $g=cg_0$ for some $c\in \F_q^*$ and
monic polynomial $g_0 \in \A$. Then, noticing $q \equiv 3 \pmod 4$, we get
$$
g^{(q-1)/2}= \pm g_0^{(q-1)/2} = (\pm g_0) ^{(q-1)/2}.
$$
Based on the above observations and the fact that $-1$ is not a quadratic residue modulo $P$ and is the unique element of order two in $(\A/P\A)^*$, we deduce that
\begin{equation}
\label{eq:identity}
1= \prod_{\substack{\textrm{$\left(\frac{g}{P}\right)=1$} \\ 0 \le \deg g < d}} g
=\prod_{\substack{\textrm{$\sgn(g)=\pm 1$, $\left(\frac{g}{P}\right)=1$} \\ 0 \le \deg g < d}} g^{(q-1)/2},
\end{equation}
where we also use the simple fact that the inverse of a quadratic residue is also a quadratic residue (modulo $P$).

Using \eqref{eq:identity}, we obtain
\begin{equation*}
\begin{split}
M(P) & = \Big(  \prod_{\substack{\textrm{$g$ monic, $\left(\frac{g}{P}\right)=1$} \\ 0 \le \deg g < d}} g \prod_{\substack{\textrm{$g$ monic, $\left(\frac{g}{P}\right)=-1$} \\ 0 \le \deg g < d}} g\Big)^{(q-1)/2} \\
& = (-1)^N     \prod_{\substack{\textrm{$g$ monic, $\left(\frac{g}{P}\right)=1$} \\ 0 \le \deg g < d}} g^{(q-1)/2} \prod_{\substack{\textrm{$g$ monic, $\left(\frac{g}{P}\right)=-1$} \\ 0 \le \deg g < d}} (-g)^{(q-1)/2} \\
& = (-1)^N     \prod_{\substack{\textrm{$\sgn(g)=\pm 1$, $\left(\frac{g}{P}\right)=1$} \\ 0 \le \deg g < d}} g^{(q-1)/2}\\
& \equiv (-1)^N \pmod P.
\end{split}
\end{equation*}
 Now, the desired result follows from \eqref{eq:nonresidue}.
\end{proof}

In the following, we extend the result in Lemma \ref{lem:MP} to the case when $f$ is a power of some prime polynomial up to a constant. Actually, we can obtain a more general form.

\begin{lemma}
\label{lem:MPe}
If $f \in \A$ is a polynomial of the form $f=aP^e$, where $a \in \F_q^{*}$, $P\in \A$ is a prime polynomial and $e$ is a positive integer, then
$$ M(f)\equiv (-1)^{\frac{(e-1)(q-1)}{2}\deg P }M(P) \pmod P. $$
\end{lemma}
\begin{proof}
Put $d= \deg P$.
Note that for any non-zero polynomial $g\in \A$ of degree $\deg g < \deg f=de$, by the Euclidean division we can write
$$
g = g_1P+g_2 \quad \textrm{for some $g_1,g_2 \in \A,$}
$$
where $g_1$ and $g_2$ satisfy $\deg g_1 <d(e-1)$ and $g_2=0$ or $0\leq \deg g_2 <d$. Then by definition, we have
$$
M(f) = \Big(  Q_{0}Q_{1} \Big)^{(q-1)/2},
$$
where
$$
Q_{0}=\prod_{\substack{g \textrm{ monic} \\ 0 \le \deg g < d}} g
\quad \textrm{and} \quad
Q_{1}=\prod_{\substack{g_2 \in \A \\ 0 \le \deg g_2 < d}}\prod_{\substack{g_1 \textrm{ monic} \\ 0 \le \deg g_1 < d(e-1),}} (g_1P+g_2).
$$
Now for $Q_{1}$, we deduce that
\begin{align*}
Q_{1} & \equiv \Big(\prod_{\substack{g \in \A \\ 0 \le \deg g < d}} g \Big)^{\frac{q^{d(e-1)}-1}{q-1}}\pmod P \\
& \equiv G(P)^{\frac{q^{d(e-1)}-1}{q-1}}\pmod P \\
& \equiv (-1)^{d(e-1)} \pmod P,
\end{align*}
where we have applied Theorem \ref{thm:con1}.
Therefore
\begin{align*}
M(f) &= \Big(  Q_{0}Q_{1} \Big)^{(q-1)/2}\\
&\equiv (-1)^{\frac{d(e-1)(q-1)}{2}}M(P) \pmod P,
\end{align*}
which completes the proof.
\end{proof}

We also need a simple but useful result, which is an analogue of \cite[Lemma 1]{CD2008}. One can prove it in a straightforward manner; we therefore omit its proof.

\begin{lemma}
\label{lem:lift}
Suppose that $q$ is odd. Let $f,g\in \A$ be two non-constant polynomials.
Given an integer $e\ge 1$, assume that $f^2 \equiv 1 \pmod{g^e}$. Then,
$f \equiv \pm 1 \pmod{g^e}$ if and only if $f \equiv \pm 1 \pmod g$, with the signs corresponding to each other.
\end{lemma}

Now we are ready to get a partial result about the value of $G(f,S)$ modulo $f$. We use the product $\delta(S)$ defined just before \eqref{eq:delta}.

\begin{theorem}
\label{thm:GPe}
Assume that $q \equiv 3 \pmod 4$.
If $f \in \A$ is a polynomial of the form $f=aP^e$, where $a \in \F_q^{*}$, $P\in \A$ is a prime polynomial of odd degree and $e$ is a positive integer, then
$$
G(f,S)\equiv (-1)^{e+\frac{1}{2}(h(-P)-3)}\delta(S) \pmod f.
$$
\end{theorem}

\begin{proof}
By \eqref{eq:GM}, we first note that
$$
G(f,S) = \delta(S)^{\frac{\Phi(f)}{q-1}}M(f) = \delta(S)M(f),
$$
where we use the fact that $\delta(S)=\pm 1$ since $q \equiv 3 \pmod 4$.
From Lemma \ref{lem:MPe}, we have
$$
M(f)\equiv (-1)^{\frac{(e-1)(q-1)}{2}\deg P }M(P) \equiv (-1)^{e-1}M(P) \pmod P,
$$
which, together with Lemma \ref{lem:MP}, gives
\begin{equation}
\label{eq:MP}
M(f)\equiv (-1)^{e-1+\frac{1}{2}(h(-P)-1)} \pmod P.
\end{equation}

Since $M(f)^2 \equiv 1 \pmod{P^e}$ by \eqref{eq:M2}, using Lemma \ref{lem:lift} we know that
$M(f) \equiv \pm 1 \pmod{P^e}$ if and only if $M(f) \equiv \pm 1 \pmod P$, with the signs corresponding to each other. Thus, by \eqref{eq:MP} we obtain
\begin{equation}
M(f)\equiv (-1)^{e+\frac{1}{2}(h(-P)-3)} \pmod{P^e}.
\end{equation}
Now, the desired result follows.
\end{proof}

\subsection{Two or more prime divisors}

In this section, we deal with the case when $f$ has more than one prime divisors. We start with a key lemma.

\begin{lemma}
\label{lem:reduction}
If $f \in \A$ is a polynomial having the prime factorization as in \eqref{eq:factor},
then
\begin{equation*}
\begin{split}
M(f)
& \equiv \left\{\begin{array}{ll}
(-1)^{\frac{q-1}{2}\deg P_{2} }P_{2}^{-\frac{\Phi(P_{1}^{e_{1}})}{2}}  \pmod{P_{1}^{e_{1}}} & \textrm{if $t=2$,}\\
\\
P_{t}^{-\frac{\Phi\big(P_{1}^{e_{1}}P_{2}^{e_{2}}\cdots P_{t-1}^{e_{t-1}}\big)}{2}}  \pmod{P_{1}^{e_{1}}P_{2}^{e_{2}}\cdots P_{t-1}^{e_{t-1}}} &\textrm{if $t\geq 3$.}
\end{array}
\right.
\end{split}
\end{equation*}
\end{lemma}

\begin{proof}
Put $\widetilde{f}=P_{1}^{e_{1}}P_{2}^{e_{2}}\cdots P_{t-1}^{e_{t-1}},d_{t}= \deg P_{t}$.
Note that for any non-zero polynomial $g\in \A$ of degree $\deg g < \deg f$, by the Euclidean division we can write
$$
g = g_1\widetilde{f}+g_2 \quad \textrm{for some $g_1,g_2 \in \A$,}
$$
where $g_1$ and $g_2$ satisfy $\deg g_1 <d_{t}e_{t}$ and $g_2=0$ or $0\leq \deg g_2 <\deg \widetilde{f}$. Then by definition, we have
\begin{equation}
\label{eq:M(f)}
M(f) = \Big(  Q_{0}Q_{1} \Big)^{(q-1)/2},
\end{equation}
where
$$
Q_{0}=\prod_{\substack{g_{2} \textrm{ monic} \\ 0 \le \deg g_{2} < \deg \widetilde{f}\\ \gcd(g_{2},f)=1}} g_{2} \ \ \
\textrm{ and}\ \ \
Q_{1}=\prod_{\substack{g_2 \in \A \\ 0 \le \deg g_2 < \deg \widetilde{f}\\ \gcd(g_{2},f)=1}}\prod_{\substack{g_1 \textrm{ monic} \\ 0 \le \deg g_1 < d_{t}e_{t},}} (g_1\widetilde{f}+g_2).
$$

Now, we define
$$
\overline{Q_{0}}=\prod_{\substack{g_{2} \textrm{ monic} \\ 0 \le \deg g_{2} < \deg \widetilde{f}\\ (g_{2},\widetilde{f})=1}} g_{2}
\quad \textrm{and} \quad
\overline{Q_{1}}=\prod_{\substack{g_2 \in \A \\ 0 \le \deg g_2 < \deg \widetilde{f}\\ \gcd(g_{2},\widetilde{f})=1}}\prod_{\substack{g_1 \textrm{ monic} \\ 0 \le \deg g_1 < d_{t}e_{t}}} (g_1\widetilde{f}+g_2),
$$
and obtain
\begin{align*}
\overline{Q_{1}}
&\equiv \Big(\prod_{\substack{g_2 \in \A \\ 0 \le \deg g_2 < \deg \widetilde{f}\\ \gcd(g_{2},\widetilde{f})=1}}g_2\Big)^{\frac{q^{d_{t}e_{t}}-1}{q-1}} \pmod{\widetilde{f}} \\
&\equiv G(\widetilde{f})^{\frac{q^{d_{t}e_{t}}-1}{q-1}} \pmod{\widetilde{f}}.
\end{align*}

For relating $Q_j$ to $\overline{Q_{j}}$, we multiply all relevant multiples of $P_t$ back to $Q_0$ and $Q_1$. More precisely, on the right-hand side of \eqref{eq:M(f)} we multiply numerator and denominator by
$$\Big(\prod_{\substack{g \textrm{ monic} \\ 0\leq \deg g < \deg \widetilde{f} +d_{t}(e_{t}-1)\\ \gcd(g,\widetilde{f})=1}} gP_{t}\Big)^{\frac{q-1}{2}},$$
which is equal to
\begin{align*}
&\Big(\prod_{\substack{g \textrm{ monic} \\ 0 \le \deg g < \deg \widetilde{f}\\ \gcd(g,\widetilde{f})=1}}gP_{t}\Big)^{\frac{q-1}{2}}\Big(\prod_{\substack{g_2 \in \A \\ 0 \le \deg g_2 <\deg \widetilde{f}\\ \gcd(g_{2},\widetilde{f})=1}}\prod_{\substack{g_1 \textrm{ monic} \\ 0 \le \deg g_1 < d_{t}(e_{t}-1),}} (g_1\widetilde{f}+g_2)P_{t}\Big)^{\frac{q-1}{2}}\\
 &\equiv M(\widetilde{f})P_{t}^{\frac{\Phi(\widetilde{f})}{2}}\left(\prod_{\substack{g_2 \in \A \\ 0 \le \deg g_2 <\deg \widetilde{f}\\ \gcd(g_{2},\widetilde{f})=1}}\Big(g_2P_{t}\Big)^{\frac{q^{d_{t}(e_{t}-1)}-1}{q-1}}\right)^{\frac{q-1}{2}} \pmod{\widetilde{f}}\\
&\equiv M(\widetilde{f})P_{t}^{\frac{\Phi(\widetilde{f})}{2}}G(\widetilde{f})^{\frac{q^{d_{t}(e_{t}-1)}-1}{2}} \pmod{\widetilde{f}}.
\end{align*}
Hence, from the above discussions, we deduce that
\begin{align*}
M(f)
& \equiv \frac{\Big(\overline{Q_{0}}\cdot\overline{Q_{1}}\Big)^{\frac{q-1}{2}}}
{M(\widetilde{f})P_{t}^{\frac{\Phi(\widetilde{f})}{2}}G(\widetilde{f})^{\frac{q^{d_{t}(e_{t}-1)}-1}{2}}} \pmod{\widetilde{f}} \\
&\equiv\frac{M(\widetilde{f})\cdot G(\widetilde{f})^{\frac{q^{d_{t}e_{t}}-1}{2}}}{M(\widetilde{f})P_{t}^{\frac{\Phi(\widetilde{f})}{2}}G(\widetilde{f})^{\frac{q^{d_{t}(e_{t}-1)}-1}{2}}} \pmod{\widetilde{f}} \\
& \equiv\frac{ G(\widetilde{f})^{\frac{1}{2}q^{d_{t}(e_{t}-1)}(q^{d_{t}}-1)}}{P_{t}^{\frac{\Phi(\widetilde{f})}{2}}} \pmod{\widetilde{f}}.
\end{align*}
By Theorem \ref{thm:con1}, this completes the proof.
\end{proof}

Now, we first address the case when $f$ has exactly two prime divisors.

\begin{theorem}
\label{thm:P1P2}
Assume that $q$ is odd, and $f \in \A$ is a polynomial having the prime factorization as in \eqref{eq:factor} with $t=2$.
Then, if one of the following two conditions holds:
\begin{enumerate}
\item $q \equiv 1 \pmod 4$,

\item $q \equiv 3 \pmod 4$ and both $\deg P_1$ and $\deg P_2$ are even,
\end{enumerate}
 we have
$$
G(f,S) \equiv \Big( \frac{P_1}{P_2} \Big)  \pmod{f};
$$
otherwise, we have
$$
G(f,S) \not\equiv \pm 1  \pmod{f}.
$$
\end{theorem}

\begin{proof}
First, since $q$ is odd and $t=2$, by \eqref{eq:delta} and \eqref{eq:GM} we have
\begin{equation}
\label{eq:GM2}
G(f,S)= \delta(S)^{\frac{\Phi(f)}{q-1}}M(f)=M(f).
\end{equation}
Put $d_j = \deg P_j, j=1,2$.
By Lemma \ref{lem:reduction}, we have
\begin{align*}
M(f) & \equiv
(-1)^{\frac{(q-1)d_2}{2} }P_{2}^{-\frac{\Phi(P_{1}^{e_{1}})}{2}}  \pmod{P_1} \\
& \equiv  (-1)^{\frac{(q-1)d_2}{2} } \Big( \frac{P_2}{P_1}\Big)^{-1}  \pmod{P_1} \\
& \equiv  (-1)^{\frac{(q-1)d_2}{2} } \Big( \frac{P_2}{P_1}\Big)  \pmod{P_1}.
\end{align*}
From \eqref{eq:M2}, we know that $M(f)^2 \equiv 1 \pmod{P_1^{e_1}}$.
So, applying Lemma \ref{lem:lift} we get
\begin{equation}
\label{eq:modP1}
M(f) \equiv  (-1)^{\frac{(q-1)d_2}{2} } \Big( \frac{P_2}{P_1}\Big)  \pmod{P_1^{e_1}}.
\end{equation}
By symmetry, we have
\begin{equation}
\label{eq:modP2}
M(f) \equiv  (-1)^{\frac{(q-1)d_1}{2} } \Big( \frac{P_1}{P_2}\Big)  \pmod{P_2^{e_2}}.
\end{equation}

Now, assume that either $q \equiv 1 \pmod 4$, or $q \equiv 3 \pmod 4$ and both $d_1$ and $d_2$ are even. Then, by the reciprocity law (Lemma \ref{lem:reciprocity}) we obtain
$$
\Big( \frac{P_2}{P_1}\Big)   = \Big( \frac{P_1}{P_2}\Big),
$$
and thus
$$
M(f) \equiv \Big( \frac{P_1}{P_2} \Big) \qquad
\textrm{(mod $P_1^{e_1}$) and (mod $P_2^{e_2}$)}.
$$
Using the Chinese Remainder Theorem, we get
\begin{equation}
M(f) \equiv \Big( \frac{P_1}{P_2} \Big) \quad \pmod f.
\end{equation}
So, by \eqref{eq:GM2} we have
$$
G(f,S) \equiv \Big( \frac{P_1}{P_2} \Big) \quad \pmod f.
$$

In all other cases, one can similarly see that the product of the right-hand sides of \eqref{eq:modP1} and \eqref{eq:modP2} is equal to $-1$.
Hence, applying again the Chinese Remainder Theorem, we can conclude the proof.
\end{proof}

Finally, the case when $f$ has more than two prime divisors is much easier.

\begin{theorem}
\label{thm:P1P2P3}
Assume that $q$ is odd. If $f \in \A$ is a polynomial having the prime factorization as in \eqref{eq:factor} with $t\ge 3$,
then we have
$$
G(f,S) \equiv 1  \pmod{f}.
$$
\end{theorem}

\begin{proof}
By Lemma \ref{lem:reduction}, we have
\begin{align*}
M(f) & \equiv
P_{t}^{-\frac{\Phi\big(P_{1}^{e_{1}}P_{2}^{e_{2}}\cdots P_{t-1}^{e_{t-1}}\big)}{2}}  \pmod{P_1} \\
& \equiv  \Big( \frac{P_t}{P_1}\Big)^{-\Phi\big(P_{2}^{e_{2}}\cdots P_{t-1}^{e_{t-1}}\big)}  \pmod{P_1} \\
& \equiv  1  \pmod{P_1}.
\end{align*}
From \eqref{eq:M2}, we know that $M(f)^2 \equiv 1 \pmod{P_1^{e_1}}$.
So, applying Lemma \ref{lem:lift} we get
$$
M(f) \equiv  1 \pmod{P_1^{e_1}}.
$$
By symmetry, we have
$$
M(f) \equiv  1  \pmod{P_j^{e_j}} \quad \textrm{for $2\le j \le t$.}
$$
So, by the Chinese Remainder Theorem we obtain
\begin{equation}
M(f) \equiv 1 \pmod f.
\end{equation}
Noticing $G(f,S)= \delta(S)^{\frac{\Phi(f)}{q-1}}M(f)=M(f)$, we complete the proof.
\end{proof}

\section*{Acknowledgements} 
The authors are very grateful to the
referee for careful reading and useful comments.
The research of the first author was supported by National Natural Science Foundation of China Grant No.11526119, and
the second author was supported by the Australian
Research Council Grant DP130100237.

\end{document}